\documentclass[a4paper,11pt]{article}

\usepackage{stmaryrd}
 
\usepackage[T1]{fontenc}
\usepackage[utf8]{inputenc}
\usepackage{lipsum}
\usepackage{amsmath, amssymb, mathtools, amsthm}
\usepackage{comment}
\mathtoolsset{showonlyrefs}

\usepackage[usenames,dvipsnames]{xcolor}
\usepackage{graphicx}
\usepackage{minitoc}
\usepackage{tikz}
\usepackage{enumerate}
\usepackage{filecontents}
\usepackage[margin=0.96 in]{geometry}
\usepackage{bbm}

\usepackage[numbers,sort&compress]{natbib}
\usepackage[colorlinks=true]{hyperref}
\hypersetup{urlcolor=blue, citecolor=blue, linkcolor=blue}
\usepackage{mathtools}
\usepackage{float}
\usepackage{xspace}
\usepackage{wrapfig}

 \usepackage{epsfig}
 \usepackage{subcaption}
 \usepackage{multirow}
 \usepackage{lineno}
 \usepackage{fullpage}
 \usepackage[normalem]{ulem} 
 \usepackage{makeidx}

\usepackage[title]{appendix}
\usepackage{dsfont}

\DeclarePairedDelimiter{\abs}{\lvert}{\rvert}
\DeclarePairedDelimiter{\norm}{\lVert}{\rVert}

\DeclarePairedDelimiter{\prt}{(}{)}

\newcommand \commentout[1] {}

\DeclareMathOperator*{\supp}{\operatorname{supp}}
\DeclareMathOperator*{\tr}{\operatorname{tr}}

\newcommand{\partialt}[1]{\dfrac{\partial#1}{\partial t}}

\DeclareMathAlphabet{\mathup}{OT1}{\familydefault}{m}{n}
\newcommand{\dx}[1]{\mathop{}\!\mathup{d} #1}

\newcommand{\discuss}[1]{{\color{red}#1}}
\theoremstyle{plain}
\newtheorem{thm}{Theorem}[section]
\newtheorem{lemma}[thm]{Lemma}

\newtheorem{corollary}[thm]{Corollary}

\theoremstyle{remark}
\newtheorem{remark}[thm]{\bf Remark}

\newcommand{\ie}{\emph{i.e.}\;}
\newcommand{\cf}{\emph{cf.}\;}

\newcommand{\ddt}{\frac{\dx{}}{\dx{t}}}

\newcommand{\grad}{\nabla}

\newcommand{\g}{_\gamma}
\newcommand{\e}{_\epsilon}
\newcommand{\ep}{_{\epsilon'}}
\newcommand{\gp}{_{\gamma'}}
\newcommand{\gmax}{g_{\mathrm{max}}}

\newcommand{\R}{\mathbb{R}}

\newcommand\blfootnote[1]{%
  \begingroup
  \renewcommand\thefootnote{}\footnote{#1}%
  \addtocounter{footnote}{-1}%
  \endgroup
}

\begin{document}

\title{Convergence rate for the incompressible limit of nonlinear diffusion-advection equations
}

\author{Noemi David\thanks{Sorbonne Universit{\'e}, Inria, CNRS, Universit\'{e} de Paris, Laboratoire Jacques-Louis Lions UMR7598, F-75005 Paris, France.} \thanks{ Dipartimento di Matematica, Universit\'a di Bologna, Italy.}
\and{Tomasz D\k{e}biec\footnotemark[1]}
\and{Beno\^it Perthame\footnotemark[1]}}

\maketitle
\begin{abstract}
The incompressible limit of nonlinear diffusion equations of porous medium type has attracted a lot of attention in recent years, due to its ability to link the weak formulation of cell-population models to free boundary problems of Hele-Shaw type. Although vast literature is available on this singular limit, little is known on the convergence rate of the solutions. In this work, we compute the convergence rate in a negative Sobolev norm and, upon interpolating with $BV$-uniform bounds, we deduce a convergence rate in appropriate Lebesgue spaces.
\end{abstract}

\vskip .4cm
\begin{flushleft}
    \noindent{\makebox[1in]\hrulefill}
\end{flushleft}
\begin{flushright}
    \noindent{\makebox[1in]\hrulefill}
\end{flushright}
\vskip .4cm

\blfootnote{{Email addresses:} {noemi.david@sorbonne-universite.fr}, {tomasz.debiec@sorbonne-universite.fr},\\ {benoit.perthame@sorbonne-universite.fr}}

\section{Introduction}
We consider the following nonlinear drift-diffusion equation
\begin{equation}\label{eq: equation n}
    \partialt{n} - \nabla \cdot (n \nabla p + n\nabla V) = n g,
\end{equation}
posed on $(0,T)\times\R^d$, $d\geq2$, where $n$ describes a population density and $p=p(n)$ is the density dependent pressure. The reaction term on the right-hand side represents the population growth rate, $g=g(t,x)$, while $V=V(t,x)$ is a chemical concentration.
The pressure is assumed to be a known increasing function of the density. We consider the following two representative examples:
\begin{alignat}{2}
        p_\gamma=P_\gamma(n)&:=\frac{\gamma}{\gamma-1} n^{\gamma-1},\;\; && \gamma >1,\label{eq: power law} \\
        \shortintertext{and}
        p_\epsilon =P_\epsilon(n)&:= \epsilon \frac{n}{1-n}, && \epsilon>0.\label{eq: singular law}
\end{alignat}
We are concerned with calculating the rate at which solutions to Eq.~\eqref{eq: equation n} converge to the so-called \emph{incompressible} (or \emph{stiff pressure}) limit, as described below. More precisely we prove the following results.
\begin{thm}[Convergence rate in $\dot{H}^{-1}$]
\label{thm: main_negative norm}
    Assume~\eqref{eq: assumptions data L1},~\eqref{eq: assumption V potential}, and~\eqref{eq: assumption a lap} (for $d=2$) and~\eqref{eq: assumption growth 2} or~\eqref{eq: assumption a grad} (for $d\geq3$). For $d=2$ assume additionally~\eqref{eq: data in d=2}. Then, for all $T>0$, there exists a unique function 
    $n_\infty\in C([0,T);L^1(\R^d))$ 
    such that the sequence $n\g$ (resp.\ $n\e$) converges, as $\gamma\to\infty$ (resp. $\epsilon\to 0$) to $n_\infty$ strongly in $L^\infty(0,T;\dot{H}^{-1}(\R^d))$ with the following rate
    \begin{equation}
    \label{eq: rate in H-1}
        \sup_{t\in[0,T]} \|n\g(t) - n_\infty(t)\|_{\dot{H}^{-1}(\R^d)} \leq\frac{C(T)}{{\gamma}^{1/2}} + \norm{n\g^0-n_\infty^0}_{\dot{H}^{-1}(\R^d)}.
    \end{equation}    
\end{thm}
\begin{thm}[Convergence rate in $L^{4/3}$]
\label{thm: main_strong norm}
    Under the assumptions of Theorem~\ref{thm: main_negative norm}, and additionally~\eqref{eq: assumptions data BV},~\eqref{eq: assumption V BV} and $g\in BV((0,T)\times\R^d)$, we also have $n_\infty\in BV((0,T)\times\R^d)$ and
    \begin{equation}
    \label{eq: rate in L4/3}
        \sup_{t\in[0,T]}\,\norm{n\g(t)-n_\infty(t)}_{L^{4/3}(\R^d)} \leq \frac{C(T)}{\gamma^{1/4}} +\norm{n\g^0-n_\infty^0}_{\dot{H}^{-1}(\R^d)}^{1/2}.
    \end{equation}    
\end{thm}

\begin{thm}
    \label{thm: relation} 
 Under the assumptions of Theorem~\ref{thm: main_negative norm}, there exists a function $p_\infty\in L^\infty((0,T)\times\R^d)$ such that, after extracting a subsequence, the sequence $p_\gamma$ converges to $p_\infty$ weakly$^*$ in $L^\infty((0,T)\times \R^d)$ and the following relation holds 
\begin{equation}
    \label{eq: relation}
    p_\infty(1 - n_\infty) =0,
\end{equation}
almost everywhere in $(0,T)\times \R^d$.
\end{thm}
The above graph relation between the limit pressure and density is well-know in the literature. In particular, when considering tumor growth models it implies that saturation holds in the region where there is a positive pressure, which is usually referred to as the region occupied by the tumor. Here we provide a new proof that does not require strong convergence of the density nor the pressure.

In fact, the limit $n_\infty$ satisfies (together with a limit pressure, $p_\infty$) a free boundary type problem, discussed shortly below, and the question of passing to this limiting problem has been vastly addressed in literature. Our contribution is to provide a new proof together with a convergence rate.

\bigskip
\noindent{\bf Motivation and previous works.} 
Models like Eq.~\eqref{eq: equation n} are well-known and commonly employed in a variety of applications, for instance in bio-mathematical modelling of living tissue.
In the case $V=0$, $g=0$, it is well-known that if the pressure satisfies the power law~\eqref{eq: power law}, then Eq.~\eqref{eq: equation n} is actually the porous medium equation
\begin{equation}
    \label{eq: PME}
    \partialt{n\g} - \Delta n\g^\gamma = 0,
\end{equation}
whose well-understood properties (e.g.\ regularising effects) facilitate the analysis notably.
The other choice of the pressure, given by Eq.~\eqref{eq: singular law}, is well-known in kinetic theory of dense gases where the short-distance interactions between particles are strongly repulsive. In this spirit it has been used in models describing collective motion or congested traffic flow, see~\cite{HechtVauchelet2017, PerrinZatorska2015, DegondHua2013, Degond_ARMA2008, Degond_M3AS2008}.
Despite having a singularity when the population density reaches its maximum value (here standardised to $1$), this choice of pressure gives rise to a tissue growth model with similar properties -- indeed, the crucial a-priori estimates are the same and the limiting free boundary type problem is almost identical. A difference is that the singularity in the pressure prevents the cell densities to ever rise above the maximum value $1$.
Taking advantage of these similarities, we shall henceforth index the solution of Eq.~\eqref{eq: equation n} by $\gamma$, $n=n\g$, and consider the singular limit $\gamma\to\infty$. Each of the assumptions and properties we discuss below has its natural $\epsilon$-analogue by putting $\epsilon = 1/\gamma$.

Let us recall that the study of the incompressible limit has a long history and it has been investigated for many different models related to Eq.~\eqref{eq: equation n}. The first result on the limit $\gamma\rightarrow\infty$ has been obtained for the classical porous medium equation~\eqref{eq: PME}.
The most interesting difference from the case with a non-trivial reaction term is that the free boundary problem arising in the limit turns out to be stationary. In fact, as proven in~\cite{CaffarelliFriedman1987} the limit density, $n_\infty$, is independent of time. This result can be intuitively explained by noticing that the degenerate diffusivity of Eq.~\eqref{eq: equation n}, namely $\gamma n^{\gamma-1}$ converges to $0$ if $n<1$, while it tends to infinity in the regions where $n>1$. Therefore, while there is no motion in the regions where the density is below 1, where the solution lies above this level it tends to collapse instantaneously, \cf~\cite{GilQuiros2003}. In the absence of reaction terms and, hence, of any evolution process in the Hele-Shaw problem, the limit pressure turns out to be constantly equal to zero, $p_\infty\equiv 0$.

Introducing non-trivial Dirichlet boundary conditions changes drastically the behaviour of the limit free boundary problem. In fact, the limit pressure no longer vanishes and this triggers the evolution of the interface in accordance with Darcy's law (which states that the velocity of the free boundary is proportional to the pressure gradient). This problem was addressed in~\cite{GilQuiros2001}, where the authors study the incompressible limit of the porous medium equation defined in  $[0,\infty)\times\Omega$, where $\Omega$ is a compact subset of $\R^d$, and the pressure satisfies $p(t,x)=f(t,x)$ on $\partial \Omega$, for some $f(t,x)\geq 0$. In the absence of Dirichlet boundary data, \ie \ $f\equiv 0$, and for $\Omega$ large enough, the problem is actually the same as in~\cite{CaffarelliFriedman1987} and it still holds that $n_\infty=n_\infty(x)$ as well as $p_\infty\equiv 0$. 
On the other hand, if one imposes the pressure to be strictly positive somewhere on $\partial\Omega$, \ie\ $f\not\equiv 0$, then the pressure gradient no longer vanishes and the dynamics of the limit problem is governed by Darcy's law.

The same non-stationary effect, although due to different dynamics, is produced by a non-trivial reaction process.
The incompressible limit for Eq.~\eqref{eq: equation n} without convective effects, \ie \ $V=0$, and with a pressure-dependent growth rate $g=G(p)$, was first addressed in the seminal paper~\cite{PQV} by Perthame, Quir\'os and V\'azquez.
They prove that it is possible to extract subsequences of $n\g$ and $p\g$ which converge in the $L^1$-norm to functions
\begin{equation}
    n_\infty \in C([0,T];L^1(\R^d)) \cap BV((0,T)\times\R^d),\quad p_\infty \in L^2(0,T;H^1(\R^d)) \cap BV((0,T)\times\R^d),
\end{equation}
satisfying the following equation in the sense of distributions on $(0,T)\times\R^d$
\begin{equation}
\label{eq: limit evolution}
    \partialt{n_\infty} - \Delta p_\infty = n_\infty G(p_\infty),
\end{equation}
\commentout{\discuss{shouldn't we just keep one of the two? since by the relation (7) they are in fact the same?}
\begin{equation}
    \partialt{n_\infty} - \nabla\cdot(n_\infty\nabla p_\infty) = n_\infty G(p_\infty),
\end{equation}}
and the following relations
\begin{equation}
    (1-n_\infty)p_\infty = 0,
\end{equation}
almost everywhere, as well as
\begin{equation}
\label{eq: complementarityPME}
    p_\infty(\Delta p_\infty + G(p_\infty)) = 0,
\end{equation}
in the sense of distributions. 
The last equality is usually referred to as the \emph{complementarity relation} and represents the link between the limit equation and the free boundary problem. In fact, denoting by $\Omega(t):=\{x\in\R^d \ | \ p_\infty(x,t)>0\}$ the region occupied by the tumor, from Eq.~\eqref{eq: complementarityPME} one can see that the pressure satisfies an elliptic equation in the evolving domain $\Omega(t)$ with homogeneous Dirichlet boundary conditions. The free boundary $\partial\Omega(t)$ is moving under Darcy's law, which finally allows to obtain the fully geometrical representation of the limit problem.
A derivation of the velocity law can be found in~\cite{PQV} for initial data given by characteristic functions of bounded sets, although the proof relies on formal arguments. A weak (distributional) and a measure-theoretic interpretation of the free boundary condition have been recovered in~\cite{MelletPerthameQuiros2017}, while in~\cite{Kpo} the same result is achieved through the viscosity solutions approach.

An analogous result regarding the limit $\gamma\to\infty$ has been shown in~\cite{HechtVauchelet2017} for the pressure law given by Eq.~\eqref{eq: singular law}. The authors obtain virtually the same limiting problem, the only difference being that the complementarity relation~\eqref{eq: complementarityPME} becomes
\begin{equation}
\label{eq: complementarity singular}
    p_\infty^2(\Delta p_\infty + G(p_\infty)) = 0,
\end{equation}
see~\cite[Theorem~2.1]{HechtVauchelet2017}.
Let us point out that due to uniform estimates in $L^\infty$ the convergence of the sequence of densities is also true in any $L^p$-space, $p<\infty$.

The Hele-Shaw limit for the porous medium equation including convective effects, \cf Eq.~\eqref{eq: equation n} with $V\not\equiv 0$, and possibly reaction terms, has attracted a lot of interest as well. Similarly as for the driftless case, when passing to the limit $\gamma\rightarrow \infty$, the model converges to a free boundary problem where, however, the interface dynamics is no longer driven only by Darcy's law, but also by the external drift, \ie \ the normal velocity is given by $-(\nabla p_\infty + \nabla V)\cdot\nu$, where $\nu$ is the outward normal direction. The asymptotics as $\gamma\rightarrow\infty$ has been addressed both for local and non-local drift, in the absence of reactions, see for instance~\cite{AlexanderKimYao2014, CraigKimYao2018}, where the authors adopt techniques relying on the gradient flow structure of the equation. In~\cite{KPW2019}, Kim, Po\v{z}\`ar and Woodhouse include also a linear reaction term into the equation and are able to prove the convergence to the incompressible limit using viscosity solutions. Recently, in~\cite{DavidSchmidtchen2021} the authors show that the complementarity condition including a drift, \ie
\begin{equation*}
    p_\infty(\Delta p_\infty +\Delta V + G(p_\infty) ) = 0,
\end{equation*}
holds in the sense of distributions. 

In recent years, many other variations of the model at hand have been proposed together with the analysis of their incompressible limit. We refer the reader to \cite{DavidPerthame2021} for a model including the effects of nutrients, \cite{GKM2020} for the generalization of the driftless model with a non-monotone proliferation term, and \cite{TVCVDP} for the model including active motion. In order to account for visco-elastic effects, several models propose to use Brinkman's law instead of Darcy's law \cite{PV2015}. Moreover, cross-reaction-diffusion model using Darcy's law, Brinkman's law or singular pressure law have attracted a lot of attention as they raise challenging questions both on the existence of solutions and their incompressible limit, see \cite{LiuXu2021, BPPS, GwiazdasPerthame, CFSS, DebiecSchmidtchen2020,DebiecEtAl2021}.

Our aim is to compute the rate of convergence of the solutions of Eq.~\eqref{eq: equation n} as $\epsilon\rightarrow 0$ or $\gamma\rightarrow \infty$ in Eq.~\eqref{eq: singular law} or Eq.~\eqref{eq: power law} respectively. To the best of our knowledge the only result in this direction is given by Alexander, Kim and Yao in~\cite{AlexanderKimYao2014} for the porous medium equation including a space-dependent drift.
Passing to the incompressible limit, the authors are able to build a link between the Hele-Shaw model and the following congested crowd motion model
\[
    \partial_t{n} + \nabla\cdot(n \nabla V) = 0, \quad \text{if} \quad n<1,
\] 
with the constraint $n\leq 1$. To prove the equivalence of the two models, they study the convergence as $\gamma \rightarrow \infty$ of the porous medium equation with drift, \cf~Eq.~\eqref{eq: equation n} with $G\equiv 0$. Unlike~\cite{PQV}, their approach is based on viscosity solutions. On the one hand, they are able to prove locally uniform convergence of the viscosity solution of Eq.~\eqref{eq: equation n} to a solution of the Hele-Shaw model. On the other hand, they show the convergence of the porous medium equation with drift to the aforementioned crowd motion model in the 2-Wasserstein distance. Therefore, they prove the equivalence of the two models in the special case of initial data given by ``patches'', namely $n^0=\mathds{1}_{\Omega_0}$ for a compact set $\Omega_0$. In fact, the locally uniform limit holds only for solutions of the form of a characteristic function, while the limit in the 2-Wasserstein metric holds for any bounded initial data, $0\leq n^0\leq 1$ with finite energy and second moment. 
Moreover, while the local uniform convergence only requires a strict subharmonicity assumption on the drift term, \ie \ $V\in C^2(\R^d),$ $\Delta V>0$, stronger regularity is needed to pass to the 2-Wasserstein limit. More precisely the authors make the following assumptions on $V=V(x)$: there exists $\lambda\in\R$ such that
\begin{equation*}
      \inf_{x\in\R^d} V(x) = 0, \qquad
      D^2 V(x) \geq \lambda I_d, \quad \forall x \in \R^d, \qquad
       \|\Delta V\|_{L^{\infty}(\R^d)}\leq C.
\end{equation*}
Under these assumptions, they derive the following rate of convergence, \cf~\cite[Theorem 4.2.]{AlexanderKimYao2014}
 \begin{equation*}
     \sup_{t\in[0,T]} W_2(n\g(t),n_\infty(t))\leq \frac{C}{\gamma^{1/24}},
 \end{equation*}
where $C$ is a positive constant depending on $\int V n^0 $, $\|\Delta V\|_\infty$ and $T$.

The main result of this paper offers an improved polynomial rate of convergence in a negative Sobolev norm and the strong topology of Lebesgue spaces, see Theorems~\ref{thm: main_negative norm} and~\ref{thm: main_strong norm} above and Corollary~\ref{cor: interpolated rates} below.
Let us remark that the 2-Wasserstein distance and the $\dot{H}^{-1}$-norm can be bounded by each other when the densities are uniformly bounded away from vacuum, see Appendix~\ref{sec: W_2 vs H-1}.
We refer the reader to~\cite[Section~5.5.2]{SantambrogioOTAM}, and references therein, for further discussion about the equivalence of the two distances.
 
\bigskip
\noindent{\bf Preliminaries and assumptions.}
Throughout this paper we make the following assumptions on the components of the model.
Firstly, we assume that Eq.~\eqref{eq: equation n} is equipped with non-negative initial data $n\g^0$ (resp. $n_\epsilon^0$) such that there is a compact set $K\subset\R^d$ and a function $n_\infty^0\in L^1(\R^d)$ satisfying 
\begin{equation}\label{eq: assumptions data L1}\tag{A--$L^1$data}
\begin{aligned}   
 \supp{n\g^0}\subset K,\;\; &p_\gamma^0=P_\gamma(n_\gamma^0)\in L^\infty(\R^d), &   &0\leq n\g^0\in L^1(\R^d), & &\norm{n\g^0 - n_\infty^0}_{L^1(\R^d)}\rightarrow 0, \\[0.3em]
&p_\epsilon^0=P_\epsilon(n_\epsilon^0)\in L^\infty(\R^d),  &&0\leq n_\epsilon^0\in L^1(\R^d), & &\norm{n_\epsilon^0 - n_\infty^0}_{L^1(\R^d)}\rightarrow 0.
\end{aligned}
\end{equation}


\noindent Note in particular that the compact support assumption is needed only in the power law pressure. This is because when the pressure is given by Eq.~\eqref{eq: singular law} we can achieve our main estimate without a uniform bound for the pressure in $L^\infty$, which is not the case for the power law. Having uniformly compactly supported data allows to derive a maximum principle for the equation satisfied by the pressure.
When additionally specified, we assume further
\begin{equation}
    \label{eq: assumptions data BV}
    n\g^0 \in BV(\R^d),\qquad  \Delta \prt*{n\g^0}^\gamma\in L^1(\R^d), \tag{A--BV\,data}
\end{equation}
uniformly in $\gamma$.
Secondly, the chemical concentration potential, $V$, is assumed to satisfy
\begin{equation}
\label{eq: assumption V potential}
   D^2V \geq \prt*{\lambda + \frac12 \tr(D^2 V)} I_d,\;\; \text{ for some $\lambda\in\R$}, \tag{A--drift}
\end{equation}
and additionally
\begin{equation}
\label{eq: assumption V BV}
          D^2 V \in   L^\infty((0,T)\times\R^d), \quad
         \grad V \in L^\infty((0,T)\times\R^d), \quad \nabla\Delta V \in L^1((0,T)\times\R^d).     \tag{A--BV drift}
\end{equation} 
Thirdly, we assume the proliferation rate, $g=g(t,x)$, to be locally integrable and satisfy one of the following assumptions
\begin{equation}
\label{eq: assumption a lap}
    g_+ \in L^\infty((0,T)\times\R^d)\,\,\text{ and }\,\, \Delta g \geq 0, \tag{A--reaction}
\end{equation}
where $f_+:=\max(f,0)$ denotes the positive part of the function,
or
\begin{equation}
\label{eq: assumption growth 2}
    |\Delta g|_- \in L^\infty(0,T; L^{d/2}(\R^d)),\;\; d\geq 3, \tag{A--reaction'}
\end{equation}
where $f_-:=\max(-f,0)$ denotes the negative part of the function, or in alternative
\begin{equation}
\label{eq: assumption a grad}
    g_+ \in L^\infty((0,T)\times\R^d)\,\,\text{ and }\,\, \nabla g \in L^\infty(0,T;L^d(\R^d)),\;\; d\geq 3. \tag{A--reaction''}
\end{equation}

\noindent
Under these assumptions one can derive several crucial uniform estimates for Eq.~\eqref{eq: equation n}.
\begin{lemma}[A-priori estimates]
\label{lemma: a-priori}
Under assumption~\eqref{eq: assumptions data L1} the family $n\g$ of solutions to Eq.~\eqref{eq: equation n} satisfies the following bounds, uniformly in $\gamma$
    \begin{enumerate}
        \item $\supp{p\g(t)}\subset K(t)$ for some compact set $K(t)$,
        \item 
        there exists a positive constant $p_M=p_M(T)$ such that $0 \le p\g \leq p_M, \; 0\leq n\g \leq ~\prt*{\frac{\gamma-1}{\gamma}p_M}^{\frac{1}{\gamma-1}},$
        \item $n\g \in L^\infty(0,T; L^1(\R^d))$.
    \end{enumerate}
    Assuming in addition \eqref{eq: assumption V BV} we also have $n\g \in L^\infty(0,T; BV(\R^d))$.
    Moreover, when the pressure is given by Eq.~\eqref{eq: singular law}, we have $0\leq n\e\leq1$.
\end{lemma}
\noindent These bounds are enough for our purposes. Their proofs are fairly standard and derived in full detail in~\cite{PQV, HechtVauchelet2017, GKM2020, DavidSchmidtchen2021}, so we omit them here. Let us point out that to fully justify passing to the incompressible limit $\gamma\to\infty$ one usually needs to derive additional estimates for the time derivative of the population density and the pressure.


\begin{remark}[More general drift term]
    It is easily seen in the proof of our main results that we do not require the drift velocity to be a gradient. Indeed, one can replace the term $n\nabla V$ in Eq.~\eqref{eq: equation n} by $n U(t,x)$ with appropriate modifications to the regularity assumptions~\eqref{eq: assumption V potential} and~\eqref{eq: assumption V BV}.
\end{remark}


\bigskip
\noindent
Our approach is to first obtain a rate of convergence in the homogeneous negative Sobolev norm $\dot{H}^{-1}$ and then interpolate with the uniform bound in $BV$ to deduce a convergence rate in Lebesgue spaces. To realise this program we make use of the diffusion structure of the problem and ``lift'' the Laplacian.
More precisely, we define the function $\varphi$ to be the solution of the following Poisson equation in $(0,T)\times\R^d$
\begin{equation}
\label{eq: phi}
   -\Delta\varphi\g = n\g,
\end{equation}
given by the convolution $\varphi\g = \mathcal{K}\star n\g$, where $\mathcal{K}$ is the fundamental solution of the Laplace equation. Explicitly, for $x\neq 0$,
\begin{equation}
    \mathcal{K}(x) = 
        \begin{dcases}
     -\frac{1}{2\pi}\ln|x|, \quad &\text{for} \quad d=2,\\[0.3em]
     \frac{1}{d(d-2)\omega_d}|x|^{2-d},  \quad &\text{for} \quad d\geq 3,
        \end{dcases}
\end{equation}
where $\omega_d$ denotes the volume of the unit ball in $\R^d$.

Suppose for now that $d\geq 3$. Then a straightforward application of Young's inequality shows that
\begin{equation}
    \varphi\g \in L^p(\R^d),\;\;\text{for }\; p>\frac{d}{d-2},
\end{equation}
and
\begin{equation}
    \nabla\varphi\g \in L^2(\R^d).
\end{equation}
If $d=2$, then we do not have $\varphi\g\in L^\infty(\R^2)$ and we cannot apply Young's inequality to deduce square-integrability of $\nabla\varphi$ (indeed, this is an endpoint case). We can however apply the logarithmic Hardy-Littlewood-Sobolev inequality, \cf\ Lemma~\ref{lemma: log-HLS}, to deduce that
\begin{equation}
    \int_{\R^2}|\nabla\varphi\g|^2 = \int_{\R^2}n\g\varphi\g \leq C,
\end{equation}
provided that $n\g\ln{n\g}$ is integrable (uniformly in $\gamma$). This can be guaranteed under additional assumptions on the initial data
\begin{equation}
\label{eq: data in d=2}
    \int_{\R^2}n\g^0\ln{n\g^0} < \infty, \quad \int_{\R^2}|x|^2n\g^0 <\infty. \tag{A--2D\,data}
\end{equation}
These bounds are propagated and imply integrability of $n\g\ln{n\g}$. We refer the reader to Appendix~\ref{sec: a-priori for d=2} for proofs of these facts.


\medskip

Notice that the $L^1$ convergence of the initial data implies the convergence of $\nabla\varphi\g^0$ to $\nabla\varphi_\infty^0$ in $L^2$. Moreover, the uniform bounds on $n\g$ together with the Hardy-Littlewood-Sobolev inequality imply that the convolution $n\g\mapsto \mathcal{K}\star n\g$ is a bounded linear operator from $L^{2d/d+2}$ to $L^2$. Therefore there is a subsequence $\nabla\varphi_{\gamma_k}$ which converges weakly in $L^2$ to $\nabla\varphi_\infty$.

Finally, we recall that the gradient $\nabla\varphi$ can be used to represent the $\dot{H}^{-1}$-norm of the function $n$ as follows
\begin{equation}
    \norm{n(t)}_{\dot{H}^{-1}(\R^d)} = \norm{\nabla\varphi(t)}_{L^2(\R^d)}.
\end{equation}

Having obtained a convergence rate in the negative norm and assuming additionally the $BV$ bounds provided by Lemma~\ref{lemma: a-priori}, we will use the following interpolation inequality, proved (in greater generality) by Cohen et al.~\cite{Cohen_2003} (see also~\cite{CintiOtto2016}), to deduce a rate in the Lebesgue $4/3$-norm:
\begin{lemma}[Interpolation inequality]
\label{lemma: interpolation}
    There exists a constant $C=C(d,T)>0$, such that, for all $t\in[0,T]$,
    \begin{equation}
    \label{eq: interpolation}
        \norm{n(t)}_{L^{4/3}(\R^d)} \leq C |n(t)|_{BV(\R^d)}^{1/2} \norm{\nabla\varphi(t)}_{L^2(\R^d)}^{1/2}.
    \end{equation} 
\end{lemma}
\noindent 
Thus, Theorem~\ref{thm: main_strong norm} is a simple consequence of Theorem~\ref{thm: main_negative norm}, Lemma~\ref{lemma: interpolation} and the uniform bound in $BV$ provided by Lemma~\ref{lemma: a-priori}.


By the usual log-convex interpolation of the $L^p$-norms we readily obtain the following corollary to Theorem~\ref{thm: main_strong norm}.
\begin{corollary}[Convergence rate in $L^p$]
\label{cor: interpolated rates}
    \begin{equation}
        \sup_{t\in[0,T]}\,\norm{n\g(t)-n_\infty(t)}_{L^{p}(\R^d)} \leq  \frac{C}{\gamma^{\alpha}},
    \end{equation}
    with
    \begin{equation}
    \alpha :=  
    \begin{dcases}
     \frac{p-1}{p}, \quad &\text{for} \quad p\in(1,4/3],\\[0.3em]
     \frac{1}{3p},  \quad &\text{for} \quad p\in[4/3,\infty).
     \end{dcases}
\end{equation}
\end{corollary}

\begin{remark}[Finite speed of propagation]
\label{rmk: finite propagation}
    When one assumes additionally that the initial data have uniformly compact support, 
    then at any later time the support of $n\g$ is still uniformly contained in a bounded set (this is one of the fundamental properties of the porous medium equation, see~\cite[Lemma~2.6]{PQV} and~\cite[Lemma~3.3]{HechtVauchelet2017} for the model with a non-zero right-hand side). Therefore one can consider problem~\eqref{eq: equation n} to be posed on a bounded subset of $\R^d$ with homogeneous Dirichlet boundary condition. Naturally our results remain true in this case with the improvement that we obtain a rate $\sim\gamma^{-1/4}$ in any $L^p$-norm, $1\leq p\leq 4/3$. In particular this covers the case of ``patches'', \ie, when the initial distribution is given by an indicator function of a compact set, as considered recently in~\cite{AlexanderKimYao2014}. 
\end{remark}



\noindent{\bf Plan of the paper.}
The remainder of the paper is devoted to proving the main theorem. It turns out that the equation can be conveniently trisected and dealt with term-by-term: considering separately the pressure-driven advection, drift, and proliferation. Indeed, it is the diffusion term that governs the rate of convergence.
The proof is therefore structured as follows. In Sections~\ref{sec: singular} and~\ref{sec: power law} we prove the main theorem for the choice of the singular pressure in Eq.~\eqref{eq: singular law} and the power law pressure in Eq.~\eqref{eq: power law} in the absence of reactions and drift. Then in Section~\ref{sec: drift-reaction} we explain how to treat the additional terms.

\noindent{\bf Notation.} Henceforth we shall usually suppress the dependence on time and space of the quantities of interest, only exhibiting the time variable in the final results. Similarly, for the sake of brevity, all space integration should be understood with respect to the $d$-dimensional Lebesgue measure.

\section{Singular pressure law}
\label{sec: singular}
In this and the following section, to explain the main idea in a simple situation, we ignore the drift and proliferation terms in Eq.~\eqref{eq: equation n} and consider only the nonlinear diffusion equation
\begin{equation}
\label{eq: simple singular}
    \partialt{n\e}-\nabla\cdot(n\e\nabla p\e) = 0,
\end{equation}
assuming now the pressure law as in Eq.~\eqref{eq: singular law}.
In this case we can rewrite Eq.~\eqref{eq: simple singular} as
\begin{equation}
    \partialt{n\e} - \Delta H\e(n\e) = 0,
\end{equation}
with
\begin{equation}\label{eq: H epsilon}
    H\e(n\e):=\int_0^{n\e} s p\e'(s)\dx{s} = \epsilon \frac{n\e}{1-n\e} +\epsilon \ln(1-n\e).
\end{equation}
Recall that we have the uniform bound $n\e<1$, so that the right-hand side above is well-defined with $\ln(1-n\e) \leq 0$. 

Let us take $\epsilon>\epsilon'>0$. We subtract the equation for $n\ep$ from the equation for $n_\epsilon$ to obtain
\begin{equation}\label{eq: ne - nep}
    \frac{\partial\prt{n\e- n\ep}}{\partial t} - \Delta \prt{H\e(n\e) -H\ep(n\ep)}= 0.
\end{equation}
Now we pose Eq.~\eqref{eq: phi} for both solutions $n\e$ and $n\ep$
\begin{equation*}
    -\Delta \varphi\e = n\e, \qquad -\Delta \varphi\ep = n\ep.
\end{equation*}
Then Eq.~\eqref{eq: ne - nep} reads
\begin{equation}
    -\Delta\frac{\partial\prt{\varphi\e - \varphi\ep}}{\partial t} - \Delta\prt*{H\e(n\e) -H\ep(n\ep)}= 0,
\end{equation}
and we test it against $\varphi\e-\varphi\ep$ to derive
\begin{equation*} 
    \frac{1}{2}\ddt \int_{\R^d} \left|\nabla\prt{\varphi\e - \varphi\ep}\right|^2 = \int_{\R^d} (n\e-n\ep)\prt{H\ep(n\ep) -H\e(n\e)}.
\end{equation*}
We now proceed to estimate the right-hand side. On the set $\{n\e>n\ep\}$ we make use of non-negativity of $H\e(n\e)$ and non-positivity of the logarithmic term in $H\ep(n\ep)$ to write
\begin{align*}
    \int_{\{n\e>n\ep\}} (n\e-n\ep)\prt{H\ep(n\ep) -H\e(n\e)} &\leq \epsilon'\int_{\{n\e>n\ep\}}(n\e-n\ep)\frac{n\ep}{1-n\ep}\leq 
    \epsilon'\int_{\{n\e>n\ep\}}n\ep.
\end{align*}
Similarly, on the complementary set $\{n\e\leq n\ep\}$ we have
\begin{align*}
    \int_{\{n\e\leq n\ep\}} (n\e-n\ep)\prt{H\ep(n\ep) -H\e(n\e)} &\leq \epsilon\int_{\{n\e\leq n\ep\}}(n\ep-n\e)\frac{n\e}{1-n\e}\leq 
    \epsilon\int_{\{n\e\leq n\ep\}}n\e.
\end{align*}
Therefore we have
\begin{align*}
      \frac{1}{2}\ddt \int_{\R^d} \left|\nabla\prt{\varphi\e - \varphi\ep}\right|^2 &\leq \epsilon \int_{\{n\e\leq n\ep\}} n\e + \epsilon' \int_{\{n\e\geq n\ep\}} n\ep\\[0.3em]
      &\leq \epsilon \|n\e (t)\|_{L^1(\R^d)} + \epsilon' \|n\ep (t)\|_{L^1(\R^d)},
\end{align*}
and since $n\e$ and $n\ep$ are uniformly bounded in $L^\infty((0,T),L^1(\R^d))$ with respect to $\epsilon$ and $\epsilon'$, we obtain
\begin{align}
      \frac{1}{2}\ddt \int_{\R^d} \left|\nabla\prt{\varphi\e - \varphi\ep}(t)\right|^2 \leq C (\epsilon + \epsilon').
\end{align}
Integrating in time on $[0,t)$ we then have
\begin{align}
      \frac{1}{2}\int_{\R^d} \left|\nabla\prt{\varphi\e - \varphi\ep}(t)\right|^2 \leq C t (\epsilon + \epsilon') + \int_{\R^d}\left|\nabla\prt{\varphi\e - \varphi\ep}(0)\right|^2.
\end{align}
It follows that the sequence $(\nabla\varphi\e)\e$ converges in the strong topology of $L^\infty((0,T),L^2(\R^d))$ to $\nabla\varphi_\infty$. 
Consequently, letting $\epsilon' \rightarrow 0$, we deduce the following rate for the convergence $n\e\to n_\infty$ in the space $\dot{H}^{-1}(\R^d)$
\begin{equation}
   \|n\e(t) - n_\infty (t)\|_{\dot{H}^{-1}(\R^d)} \leq C\sqrt{t} \sqrt{\epsilon} + \|n\e^0 - n_\infty^0\|_{\dot{H}^{-1}(\R^d)},
\end{equation}
where $C$ is a positive constant defined as follows
\begin{equation*}
    C= \sqrt{2 \sup_{\epsilon>0}\|n\e\|_{L^1((0,T)\times\R^d))}}.
\end{equation*}
\noindent
Assuming the additional $BV$ bounds for the initial data, we get from Lemma~\ref{lemma: a-priori} that $n\e$ is uniformly bounded in $L^\infty(0,T;BV(\R^d))$, and we can use Eq.~\eqref{eq: interpolation} to obtain the rate $\epsilon^{1/4}$, as announced in Eq.~\eqref{eq: rate in L4/3}. Thus Theorems~\ref{thm: main_negative norm} and~\ref{thm: main_strong norm} are proved in this special case.

\section{Power law}
\label{sec: power law}
Let us now consider Eq.~\eqref{eq: simple singular} with the pressure law given by Eq.~\eqref{eq: power law} and demonstrate that the method employed in the previous section remains valid. We now have the porous medium equation
\begin{equation}
    \partialt{n\g} - \Delta n\g^\gamma = 0.
\end{equation}
Let us recall that there exists a positive constant $p_M$ such that
\begin{equation*}
    0\leq \frac{\gamma}{\gamma-1} n\g^{\gamma-1} \leq p_M, \qquad 0 \leq \frac{\gamma'}{\gamma'-1} n\gp^{\gamma'-1} \leq p_M.
\end{equation*}
Let us define 
\[
    c_\gamma := \prt*{\frac{\gamma-1}{\gamma}}^{\frac{1}{\gamma-1}}p_M^{1/(\gamma-1)} \quad \text{and} \quad \tilde{n}\g := \frac{n\g}{c\g}.
\]
Then it immediately follows that $\tilde{n}\g \leq 1$ and solves the equation
\begin{equation*}
    \partial_t \tilde{n}\g - \Delta(c\g^{\gamma-1} \tilde{n}\g^\gamma)=0.
\end{equation*}
Following the same argument as before, we define $\varphi\g$ and $\tilde{\varphi}\g$ by
\begin{equation*}
   -\Delta \varphi\g = n\g, \qquad  -\Delta \tilde{\varphi}\g = \tilde{n}\g, 
\end{equation*}
\ie $\tilde{\varphi}\g = \varphi\g / c\g$. 

Without loss of generality, we take $1<\gamma<\gamma'$. 
Now we subtract the equation for $\tilde{n}\gp$ from the equation for $\tilde{n}\g$ to obtain
\begin{equation}\label{eq: ng - ngp new}
    \frac{\partial\prt{\tilde{n}\g- \tilde{n}\gp}}{\partial t} - \Delta \prt{c\g^{\gamma-1}\tilde{n}\g^\gamma - c\gp^{\gamma'-1} \tilde{n}\gp^{\gamma'}}= 0.
\end{equation}
Then from Eq. \eqref{eq: ng - ngp new} we have
\begin{equation*} 
   -\Delta \frac{\partial\prt{\tilde\varphi\g - \tilde\varphi\gp}}{\partial t}-\Delta \prt{  c\g^{\gamma-1}\tilde{n}\g^\gamma -c\gp^{\gamma'-1}\tilde{n}\gp^{\gamma'}}= 0,
\end{equation*}
and we test it against $\tilde\varphi\g - \tilde\varphi\gp$ to deduce 
\begin{equation*}
\label{eq: PME terms}
\begin{split}
    \frac{1}{2}\ddt \int_{\R^d} \left|\nabla\prt{\tilde\varphi\g - \tilde\varphi\gp}\right|^2 &= \int_{\R^d} \prt{c\g^{\gamma-1}\tilde{n}\g^\gamma -c\gp^{\gamma'-1}\tilde{n}\gp^{\gamma'}}(\tilde{n}\gp-\tilde{n}\g)\\
    &\leq \int_{\R^d} c\g^{\gamma-1}\tilde{n}\g^\gamma (1-\tilde{n}\g) + \int_{\R^d} c\gp^{\gamma'-1}\tilde{n}\gp^{\gamma'} (1-\tilde{n}\gp).
    \end{split}
\end{equation*}
It is easy to see that for $0\leq s \leq 1$ it holds $s^\gamma (1-s) \leq \frac{s}{\gamma}$. Hence, we have
\begin{align*}
     \frac{1}{2}\ddt \int_{\R^d} \left|\nabla\prt{\tilde\varphi\g - \tilde\varphi\gp}\right|^2
     &\leq c\g^{\gamma-1} \frac 1 \gamma \int_{\R^d} \tilde{n}\g +  c\gp^{\gamma'-1} \frac{1}{\gamma'} \int_{\R^d} \tilde{n}\gp\\
    &\leq \prt*{\frac{\gamma-1}{\gamma} p_M  \sup_{\gamma}\|\tilde{n}\g(t)\|_{L^1(\R^d)}} \frac 1 \gamma + \prt*{\frac{\gamma'-1}{\gamma} p_M  \sup_{\gamma'}\|\tilde{n}\gp(t)\|_{L^1(\R^d)}} \frac{1}{\gamma'}\\
    &\leq  C \prt*{\frac 1 \gamma + \frac{1}{\gamma'}},
\end{align*}
where in the last inequality we used the fact that by Lemma~\ref{lemma: a-priori} $n\g$ is uniformly bounded in $L^\infty(0,T;L^1(\R^d))$. 
Finally, we remove the scaling using the triangle inequality
\begin{align*}
    \|\nabla(\varphi\g -\varphi\gp)&(t)\|_{L^2(\R^d)}^2\\[0.3em]
    &\leq  \|\nabla(\varphi\g - \tilde\varphi\g)(t)\|_{L^2(\R^d)}^2+ \|\nabla(\tilde\varphi\gp -\varphi\gp)(t)\|_{L^2(\R^d)}^2+ \|\nabla(\tilde\varphi\g -\tilde\varphi\gp)(t)\|_{L^2(\R^d)}^2\\[0.3em]
    &\leq \left|1-\frac{1}{c\g}\right|^2 \|\nabla \varphi\g(t)\|_{L^2(\R^d)}^2 + \left|1-\frac{1}{c\gp}\right|^2 \|\nabla \varphi\gp(t)\|_{L^2(\R^d)}^2\\[0.3em]
   & \qquad + C t\prt*{ \frac 1 \gamma +  \frac{1}{\gamma'}} +  \|\nabla(\tilde\varphi\g -\tilde\varphi\gp)(0)\|_{L^2(\R^d)}^2\\
   &\leq \frac{1}{\gamma} \prt*{C t + \gamma  \left|1-\frac{1}{c\g}\right|^2 \sup_\gamma \|n\g(t)\|_{\dot{H}^{-1}(\R^d)}} \\
   &\qquad + \frac{1}{\gamma'} \prt*{C t + \gamma'  \left|1-\frac{1}{c\gp}\right|^2 \sup_{\gamma'} \|n\gp(t)\|_{\dot{H}^{-1}(\R^d)}} +  \|\nabla(\tilde\varphi\g -\tilde\varphi\gp)(0)\|_{L^2(\R^d)} ^2.
\end{align*}
By the definition of $c\g$, $\gamma \left|1-\frac{1}{c\gp}\right|^2 \to 0$ as $\gamma \to \infty$. Thus, we have
\begin{equation*}
     \|\nabla(\varphi\g -\varphi\gp)(t)\|_{L^2(\R^d)}^2\leq (C t + C)\left(\frac 1 \gamma + \frac{1}{\gamma'}\right) +  \|\nabla(\tilde\varphi\g -\tilde\varphi\gp)(0)\|_{L^2(\R^d)}^2.
\end{equation*}
By the same argument, we find
\begin{align*}
    \|\nabla(\tilde\varphi\g -\tilde\varphi\gp)(0)\|_{L^2(\R^d)}^2 
    \leq C\left(\frac 1 \gamma + \frac{1}{\gamma'}\right) + \|\nabla(\varphi\g-\varphi\gp)(0)\|_{L^2(\R^d)}^2.
\end{align*}
Finally, we conclude
\begin{equation}
\label{eq: weak rate for PME}
     \|\nabla(\varphi\g -\varphi\gp)(t)\|_{L^2(\R^d)}^2\leq (C t + C)\left(\frac 1 \gamma + \frac{1}{\gamma'}\right) +  \|\nabla(\varphi\g -\varphi\gp)(0)\|_{L^2(\R^d)}^2.
\end{equation}

\noindent
Consequently, arguing as before and letting $\gamma' \rightarrow \infty$, we find
\begin{equation}
    \|n\g(t) - n_\infty (t)\|_{\dot{H}^{-1}(\R^d)} \leq \frac{C\sqrt{t}+C}{\sqrt{\gamma}} + \|n\g^0 - n_\infty^0\|_{\dot{H}^{-1}(\R^d)}.
\end{equation}
Again, under the additional $BV$ assumptions we obtain~\eqref{eq: rate in L4/3} thanks to the interpolation inequality.

\section{Including drift and reaction terms}
\label{sec: drift-reaction}
Having obtained the announced rate of convergence due to the nonlinear diffusion term, we now exhibit that we can include the drift and reaction terms. In fact, due to our assumptions on the proliferation rate and the chemical potential, all the additional terms will either have an appropriate sign, or be absorbed into the $L^2$-norm of the potential $\varphi$. We now write Eq.~\eqref{eq: equation n} as follows
\begin{equation}
    \label{eq: equation n with drift-reaction}
    \partialt{n\g} - \Delta A\g(n\g) = \nabla \cdot (n\g\nabla V) + n\g g,
\end{equation}
where $g=g(t,x)$ and $A\g$ is chosen appropriately depending on the state law for the pressure.
As seen before, there is no harm in assuming the uniform bound $n\leq1$. Then, arguing in the same way as previously, we obtain
\begin{align*}
    \frac{1}{2}\ddt &\int_{\R^d} \left|\nabla\prt{\varphi\g - \varphi\gp}\right|^2 + \int_{\R^d} (n\g-n\gp)(A\g(n\g)-A\gp(n\gp))\\[0.3em] 
    &= -\int_{\R^d}(n\g-n\gp)\nabla(\varphi\g-\varphi\gp)\cdot\nabla V + \int_{\R^d}g(t,x)(n\g-n\gp)(\varphi\g-\varphi\gp)\\[0.3em]
    &=\int_{\R^d}\Delta\prt*{\varphi\g-\varphi\gp}\nabla(\varphi\g-\varphi\gp)\cdot\nabla V
    -\int_{\R^d}g(t,x)\Delta(\varphi\g-\varphi\gp)(\varphi\g-\varphi\gp).
\end{align*}
It only remains to consider the two new terms on the right-hand side. For the first one we can write
\begin{align*}
    \int_{\R^d}&\Delta\prt*{\varphi\g-\varphi\gp}\nabla(\varphi\g-\varphi\gp)\cdot\nabla V\\[0.3em]
    &=-\int_{\R^d}\nabla(\varphi\g-\varphi\gp)^{T}D^2(\varphi\g-\varphi\gp)\nabla V -\int_{\R^d}\nabla(\varphi\g-\varphi\gp)^{T}D^2V\nabla(\varphi\g-\varphi\gp) \\[0.3em]
    &=-\frac12\int_{\R^d}\nabla|\nabla(\varphi\g-\varphi\gp)|^2\cdot\nabla V-\int_{\R^d}\nabla(\varphi\g-\varphi\gp)^{T}D^2V\nabla(\varphi\g-\varphi\gp)\\[0.3em] 
    &=\frac12\int_{\R^d}|\nabla(\varphi\g-\varphi\gp)|^2\Delta V -\int_{\R^d}\nabla(\varphi\g-\varphi\gp)^{T}D^2V\nabla(\varphi\g-\varphi\gp)\\[0.3em]
    &\leq -\lambda\int_{\R^d}|\nabla(\varphi\g-\varphi\gp)|^2, 
\end{align*}
where we have integrated by parts and used assumptions~\eqref{eq: assumption V potential}.
For the remaining term we integrate by parts to obtain
\begin{align*}
    \int_{\R^d} & g\abs*{\nabla(\varphi\g-\varphi\gp)}^2 + \int_{\R^d} (\varphi\g-\varphi\gp)\nabla(\varphi\g-\varphi\gp)\cdot\nabla g \\
    &\leq \norm{g_+}_{L^\infty((0,T)\times\R^d)}\int_{\R^d}\abs*{\nabla(\varphi\g-\varphi\gp)}^2 + \underbrace{\int_{\R^d}(\varphi\g-\varphi\gp)\nabla(\varphi\g-\varphi\gp)\cdot\nabla g}_{\mathcal{A}}.
\end{align*}
In case of $d=2$, we suppose that $g$ satisfies Assumption~\eqref{eq: assumption a lap}. Then we can integrate by parts in the last term to obtain
\begin{equation}
    \mathcal{A} = -\frac12\int_{\R^d} |\varphi\g-\varphi\gp|^2 \Delta g \leq 0.
\end{equation}

\noindent If instead $d\geq3$, we may alternatively assume that $g$ satisfies Assumption~\eqref{eq: assumption growth 2} or Assumption~\eqref{eq: assumption a grad}.
In the first case, using successively the inequalities of H\"older and Sobolev we obtain
\begin{equation*}
    \mathcal{A} \leq \frac{1}{2} \|\varphi\g - \varphi\gp\|^2_{L^{2^*}(\R^d)} \||\Delta g|_-\|_{L^{d/2}(\R^d)}\leq C_S  \||\Delta g|_-\|_{L^{d/2}(\R^d)} \int_{\R^d} |\nabla(\varphi\g - \varphi\gp)|^2,
\end{equation*}
where $C_S$ denotes the constant from Sobolev inequality, and $2^* = \frac{2d}{d-2}$ is the Sobolev conjugate exponent.
Otherwise, if $g$ satisfies Eq.~\eqref{eq: assumption a grad}, in order to estimate the term $\mathcal{A}$ we do not integrate it by parts and we use in turn the inequalities of Young, H\"older and Sobolev to obtain
\begin{align*}
    2\mathcal{A} &\leq \int_{\R^d}\abs*{\nabla(\varphi\g-\varphi\gp)}^2 + \int_{\R^d}\abs*{(\varphi\g-\varphi\gp)}^2|\nabla g|^2 \\[0.3em]
    &\leq \int_{\R^d}\abs*{\nabla(\varphi\g-\varphi\gp)}^2 + \norm{\varphi\g-\varphi\gp}_{L^{2^*}(\R^d)}^2\norm{\nabla g}_{L^d(\R^d)}^2 \\[0.3em]
    &\leq \prt*{1 + C_S\norm{\nabla g}_{L^d(\R^d)}^2}\int_{\R^d}\abs*{\nabla(\varphi\g-\varphi\gp)}^2.
\end{align*}
\noindent
Therefore we have
\begin{align*}
    \frac{1}{2}\ddt \int_{\R^d} \left|\nabla\prt{\varphi\g - \varphi\gp}\right|^2 + \int_{\R^d} (n\g-n\gp)(A\g(n\g)-A\gp(n\gp))
    \leq C\int_{\R^d} \left|\nabla\prt{\varphi\g - \varphi\gp}\right|^2.
\end{align*}
Assuming for concreteness the power law pressure, using inequality~\eqref{eq: weak rate for PME} and a Gronwall inequality, we deduce
\begin{equation}
    \sup_{t\in[0,T]}\norm{\nabla(\varphi_\gamma - \varphi_{\gamma'})(t)}_{L^2(\R^d)} \leq C\prt*{\frac{1}{\sqrt{\gamma}}+\frac{1}{\sqrt{\gamma'}}} + \norm{\nabla(\varphi_\gamma - \varphi_{\gamma'})(0)}_{L^2(\R^d)}.
\end{equation}
Finally, passing to the limit $\gamma'\to\infty$, we conclude the proof of Theorem~\ref{thm: main_negative norm}. Using the uniform $BV$-bound and Eq.~\eqref{eq: interpolation} we obtain Theorem~\ref{thm: main_strong norm}.

\subsection{Limit relation between $n_\infty$ and $p_\infty$}
Here we prove relation \eqref{eq: relation} between the limit density and pressure, where $p_\infty$ is defined as the weak$^*$ limit (up to a sub-sequence) of $p_\gamma$ in $L^\infty((0,T)\times\R^d)$. 
\begin{proof}[Proof of Theorem~\ref{thm: relation}]
The relation is a straightforward consequence of the main estimate obtained in Section~\ref{sec: power law}.
We inspect Eq.~\eqref{eq: PME terms}, this time not ignoring the non-positive terms.
After integration in time, these terms can be bounded as follows, using Eq.~\eqref{eq: weak rate for PME}
\begin{equation*}
   \int_0^T\!\!\int_{\R^d} \tilde{n}\gp^{\gamma'}(1-\tilde n\g)c\gp^{\gamma'-1}  + \int_0^T\!\!\int_{\R^d} \tilde{n}\g^{\gamma}(1-\tilde n\gp ) c\g^{\gamma-1}  \leq C(T)\prt*{\frac{1}{\gamma}+ \frac{1}{\gamma'}} + \int_{\R^d} \left|\nabla\prt{\varphi\g - \varphi\gp}(0)\right|^2.
\end{equation*}
Now let $\psi$ be a compactly supported test function and consider the quantity
\begin{align*}
    \abs*{\int_0^T\!\!\int_{\R^d}\psi \tilde n\g^\gamma(1-\tilde n\gp)} \leq \norm{\psi}_{\infty} \int_0^T\!\!\int_{\supp{\psi}} \tilde n\g^\gamma(1-\tilde n\gp) = \norm{\psi}_{\infty}\int_0^T\!\!\int_{\supp\psi} \tilde p\g^{\frac{\gamma}{\gamma-1}}(1-\tilde n\gp).
\end{align*}
Using weak lower semicontinuity of convex functionals and weak$^*$ convergence of the pressure and the density, we can pass to the limit with $\gamma'$ and $\gamma$ in turn to obtain
\begin{equation*}
    \int_0^T\!\!\int_{\R^d} \psi p_\infty (1- n_\infty) = 0,
\end{equation*}
which concludes the proof.
\qedhere
\end{proof}

\section{Conclusions and open problems}
We computed the rate of convergence of the solutions of a reaction-advection-diffusion equation of porous medium type in the incompressible limit. Our result in a negative Sobolev's norm can be interpolated with uniform $BV$-estimates in order to find a rate in any $L^p$-space for $1<p<\infty$.

How to assess the accuracy of our estimate remains an open problem. For the pure porous medium equation it might seem tempting to attempt a calculation for the illustrious example of the Barenblatt solution (taking as initial data the solution at some time $t>0$). However, a direct calculation shows that in this case the data is ``ill prepared'' in the sense that it converges (in $L^1$) to its limit profile with too slow a rate of $\sim\ln{\gamma}/\gamma$.
It is unclear how to approach the question of optimality in general.
We expect that the ``worst'' rate would be exhibited by a \textit{focusing solution}, whose support is initially contained outside of a compact set and closes up in finite time, thus generating a singularity.

Another challenging problem is to find an estimate for the convergence rate of the pressure, for which the method used above seems inapplicable as it is not clear how to relate the quantities $p\g-p\gp$ and $\varphi\g-\varphi\gp$. Consequently, we are also currently unable to treat more general, pressure dependent, reaction terms. Finally, it would be of interest to investigate whether it is possible to strengthen the estimate of Theorem~\ref{thm: main_negative norm} to Lebesgue norms without interpolation with $BV$. One advantage of any such alternative approach could be to allow for passing to the incompressible limit when $BV$ bounds are not available, as is the case for systems of equations like~\eqref{eq: equation n}. Additionally, it could allow for estimating the rate of convergence in the $L^1$-norm rather than the seemingly arbitrary $L^{4/3}$-norm.


\section*{Acknowledgements}
This project has received funding from the European Union's Horizon 2020 research and innovation program under the Marie Skłodowska-Curie (grant agreement No 754362). It has also received funding from European Research Council (ERC) under the European Union’s Horizon 2020 research and innovation program (grant agreement No 740623). T.D.\ was partially supported by National Science Center (Poland), grant number 2018/31/N/ST1/02394, and the Foundation for Polish Science (FNP).

 
\begin{appendices}

\section[Bounding the 2-Wasserstein norm by the dual Sobolev norm]{Bounding $W_2$-norm by the $\dot{H}^{-1}$-norm}
\label{sec: W_2 vs H-1}

We consider here the conservative case of Eq.~\eqref{eq: equation n}, assuming $\int n\g(t) = \int n_\infty(t)=1$.
We put $\dx{\mu\g}=n\g(x)\dx{x}$, $\dx{\mu_\infty}=n_\infty(x)\dx{x}$, ignoring time-dependence for the sake of brevity.
Furthermore we make the additional assumption that $n_\infty\geq\underline{n}>0$ for some constant $\underline{n}$.

Consider the curve $\rho:[0,1]\to\mathcal{P}_2(\R^d)$ given by $\tau\mapsto\rho_\tau:=(1-\tau)\mu\g + \tau\mu_\infty$ together with the vector field
\begin{equation}
    V_\tau(x) = \frac{1}{(1-\tau)n\g(x)+\tau n_\infty(x)}\nabla(\varphi\g-\varphi_\infty).
\end{equation}
For any test function $\psi\in C_{c}^\infty((0,1)\times\R^d)$ we have
\begin{align}
        \int_0^1\int_{\R^d}\frac{\partial\psi}{\partial\tau}\dx{\rho_\tau}(x)\dx{\tau} &= \int_0^1\int_{\R^d}\frac{\partial\psi}{\partial\tau}((1-\tau)n\g(x)+\tau n_\infty(x))\dx{x}\dx{\tau}\\
        &=\int_0^1\int_{\R^d}\psi(n\g(x)-n_\infty(x))\dx{x}\dx{\tau}\\
        &=\int_0^1\int_{\R^d}\nabla\psi\cdot\nabla(\varphi\g-\varphi_\infty)\dx{x}\dx{\tau}\\
        &=\int_0^1\int_{\R^d}\nabla\psi\cdot V_\tau\dx{\rho_\tau}(x)\dx{\tau}.
\end{align}
Therefore the pair $(\rho, V)$ solves the continuity equation
\begin{equation}
    \frac{\partial\rho_{\tau}}{\partial\tau} + \nabla\cdot(V_\tau(x)\rho_{\tau}) = 0,
\end{equation}
posed on $(0,1)\times\R^d$ with the marginal constraints
\begin{equation}
    \rho_0 = \mu\g,\quad \rho_1 = \mu.
\end{equation}
Consequently, from Theorem~5.15 in~\cite{SantambrogioOTAM}, we deduce that $\rho$ is absolutely continuous and the following inequality holds
\begin{equation*}
    |\rho'|(\tau) \leq \norm{V_\tau}_{L^2(\R^d,\dx{\rho_\tau})},
\end{equation*}
where $|\rho'|$ denotes the metric derivative of the curve $\rho$ with respect to the Wasserstein distance.
Furthermore, since $(\mathcal{P}_2(\R^d),W_2)$ is a length space, we have
\begin{equation}
     W_2(\mu\g,\mu_\infty) \leq \int_0^1 |\rho'|(\tau) \dx{\tau}.
\end{equation}
Combining these last two inequalities, we obtain the following bound
\begin{align}
    W_2(\mu\g,\mu_\infty) &\leq \int_0^1\|V_\tau(x)\|_{L^2(\R^d,\dx{\rho_\tau})}\dx{\tau}\\
    &\leq \frac{1}{\sqrt{\underline{n}}}\norm{\nabla(\varphi\g-\varphi_\infty)}_{L^2(\R^d)}\int_0^1\frac{1}{\sqrt{\tau}}\dx{\tau}\\
    &=\frac{2}{\sqrt{\underline{n}}}\norm{n\g-n_\infty}_{\dot{H}^{-1}(\R^d)}.
\end{align}

Interestingly, a reverse bound can also be shown. Rather than a positive lower bound, a common upper bound is now required of all the densities (which is of course the case here).
Let now $\sigma:[0,1]\to\mathcal{P}_2(\R^d)$ be a constant-speed geodesic from $\mu\g$ to $\mu_\infty$ and $E$ be a vector field such that $(\sigma, E)$ satisfy the continuity equation, and $\norm{E_\tau}_{L^2(\R^d;\sigma_\tau)} = W_2(\mu\g,\mu_\infty)$.
Then
\begin{align*}
    \norm{\nabla\varphi\g-\nabla\varphi_\infty}_{L^2}^2 & = \int_{\R^d} (\varphi\g-\varphi_\infty)(n\g-n_\infty)\\[0.3em]
    &=\int_0^1\int_{\R^d} \nabla(\varphi\g-\varphi_\infty)\cdot E_\tau \dx{\rho_\tau}\dx{\tau}\\[0.3em]
    &\leq \frac12\norm{\nabla\varphi\g-\nabla\varphi_\infty}_{L^2}^2 + \frac12 W_2(\mu\g,\mu_\infty)^2.
\end{align*}
We refer the reader to~\cite[Section~5.5.2]{SantambrogioOTAM}, and references therein, for further discussion about the equivalence of the two distances.

\section[Additional estimates in d=2]{Additional estimates in $d=2$}
\label{sec: a-priori for d=2}
To fully justify our main results in two dimensions, we need to derive additional estimates to guarantee that $\nabla\varphi$ is square-integrable. As mentioned in the introduction this is achieved by propagation of entropy and the second moment of the density. 

\begin{lemma}
Let $n$ be the solution of Eq.~\eqref{eq: equation n} with initial data $n^0\geq 0$ satisfying assumptions~\eqref{eq: data in d=2}. There exists a positive constant $C$, depending only on $T$, such that
\begin{enumerate}[i.]
    \item $\int_{\R^2} |x|^2 \ n(t) \leq C$, for all $t\in[0,T]$,
    \item $\int_{\R^2} n(t)\ln n(t) \leq C$, for all $t\in[0,T]$,
    \item $n|\ln n|\in L^\infty(0,T;L^1(\R^2))$.
\end{enumerate}
\end{lemma}

\begin{proof}
Let $\gmax$ denote a uniform upper bound for the reaction term $g$, and $n_H$ a uniform upper bound for the density.
Let us recall the equation for $n$, \ie
\begin{equation*}
    \partial_t n = \nabla \cdot (n \nabla p + n \nabla V).
\end{equation*}
Upon multiplying by $|x|^2/2$, integrating in space and using Young's inequality, we obtain
\begin{align*}
    \ddt \int_{\R^2} \frac{|x|^2}{2} n 
    &= -  \int_{\R^2} x \cdot (n \nabla p + n \nabla V) + \int_{\R^d} \frac{|x|^2}{2}ng\\
    &\leq  (1+\gmax)\int_{\R^2} |x|^2 n +  \frac{1}{2}  \int_{\R^2} n |\nabla p|^2 + \frac{1}{2}  \int_{\R^2} n |\nabla V|^2.
\end{align*}
Hence
\begin{align*}
    \ddt \int_{\R^2} |x|^2 n \leq  2(1+\gmax)\int_{\R^2} |x|^2 n +   n_H \prt*{\int_{\R^2} |\nabla p|^2 +  \int_{\R^2} |\nabla V|^2}.
\end{align*}
By Gronwall's lemma we infer
\begin{align*}
   \int_{\R^2} |x|^2 n(t) \leq  e^{Ct}\int_{\R^2} |x|^2 n^0 +   n_H \int_0^t e^{C(t-s)} \prt*{\|\nabla p(s)\|^2_{L^2(\R^2)} +  \|\nabla V(s)\|^2_{L^2(\R^2)} }\dx{s},
\end{align*}
and since both $\nabla p$ and $\nabla V$ are bounded in $L^2((0,T)\times \R^2)$, we conclude that
\begin{align*}
  \int_{\R^2} |x|^2 n(t) \leq C(T),
\end{align*}
for all $t\in [0,T].$

Assume now additionally that the initial entropy is finite and multiply the equation by $(n \ln n)'$ to obtain
\begin{align*}
  \ddt  \int_{\R^2} n\ln n 
  &= - \int_{\R^2} p'(n) |\nabla n|^2 - \int_{\R^2} \nabla n \cdot \nabla V + \int_{\R^2} (1+\ln n)n\, g\\
  &\leq - \int_{\R^2} \nabla n \cdot \nabla V+ \int_{\R^2} (1+\ln n)n\, g\\
  &= \int_{\R^2} n (\Delta V + g) + \gmax\int_{\R^2} n\ln n + \int_{\R^2} n\ln n\, (g-\gmax).
\end{align*}
We can control the very last term as follows.
\begin{align}
    \int_{\R^2} n\ln n\, (g-\gmax) &\leq \int_{\{n<1\}} n|\ln n|\, (\gmax-g)\\
    &= \int_{\{n<e^{-|x|^2}\}}n|\ln n|\, (\gmax-g) + \int_{{\{e^{-|x|^2}<n<1\}}}n|\ln n|\, (\gmax-g)\\
    &\leq 2\gmax \int_{\R^2}e^{-|x|^2/2} + 2\gmax\int_{\R^2}|x|^2n \leq C.
\end{align}
Hence, since $\Delta V \in L^1((0,T)\times\R^2)$, we have
\begin{equation*}
    \ddt \int_{\R^2} n\ln n \leq \gmax\int_{\R^2} n \ln n + C, 
\end{equation*}
and the result follows by Gronwall's lemma.
The final part of the lemma follows easily from the preceding two.
Indeed it is enough to show that $n|\ln{n}|$ has finite integral on the set $
\{n\leq 1\}$, and this is done exactly as above
\[
\int_{\{n<1\}} n|\ln n| \leq \int_{\R^2}e^{-|x|^2/2} + \int_{\R^2}|x|^2n < \infty.
\]
%
%
\end{proof}

Now we recall the logarithmic Hardy-Littlewood-Sobolev inequality. 
    \begin{lemma}[Logarithmic Hardy-Littlewood-Sobolev inequality]
        \label{lemma: log-HLS}
  Let $n$ be nonnegative, with $n\in L^\infty(0, T; L^1(\R^d))$ and $n\ln n \in L^\infty(0,T; L^1(\R^d))$.
  If $\int_{\R^2}n = M$, then
    \begin{equation}\label{eq: hls inequality}
        \int_{\R^2} n \ln n  + \frac{2}{M} \int_{\R^2}  \int_{\R^2} n(x) n(y) \ln|x-y| \ge -C(M),
    \end{equation}
    with $C(M) = M(1+\ln\pi - \ln M)$.   
    \end{lemma}
Recalling that $\varphi = -\frac{1}{2\pi}\ln|\cdot|\star n$, the inequality implies 
    \begin{equation*} 
     \int_{\R^2} |\nabla \varphi|^2 =   \int_{\R^2} n\varphi \leq \frac{M}{2} C +\frac{M}{2} \int_{\R^2} n \ln n < \infty.
    \end{equation*}

\end{appendices}

\bibliographystyle{abbrv}
\bibliography{TDebiec_bib}

\begin{thebibliography}{10}

\bibitem{AlexanderKimYao2014}
D.~Alexander, I.~Kim, and Y.~Yao.
\newblock Quasi-static evolution and congested crowd transport.
\newblock {\em Nonlinearity}, 27(4):823--858, 2014.

\bibitem{Degond_M3AS2008}
F.~Berthelin, P.~Degond, V.~L. Blanc, S.~Moutari, M.~Rascle, and J.~Royer.
\newblock A traffic-flow model with constraints for the modelling of traffic
  jams.
\newblock {\em Mathematical Models and Methods in Applied Sciences},
  18(supp01):1269--1298, aug 2008.

\bibitem{Degond_ARMA2008}
F.~Berthelin, P.~Degond, M.~Delitala, and M.~Rascle.
\newblock A model for the formation and evolution of traffic jams.
\newblock {\em Archive for Rational Mechanics and Analysis}, 187(2):185--220,
  nov 2007.

\bibitem{BPPS}
F.~Bubba, B.~Perthame, C.~Pouchol, and M.~Schmidtchen.
\newblock Hele-{S}haw limit for a system of two reaction-(cross-)diffusion
  equations for living tissues.
\newblock {\em Arch. Rational. Mech. Anal.}, 236:735--766, 2020.

\bibitem{CaffarelliFriedman1987}
L.~A. Caffarelli and A.~Friedman.
\newblock Asymptotic behavior of solutions of $u_t= {{\Delta}} u^m$ as $m\to
  \infty$.
\newblock {\em Indiana University Mathematics Journal}, 36(4):711--728, 1987.

\bibitem{CFSS}
J.~A. Carrillo, S.~Fagioli, F.~Santambrogio, and M.~Schmidtchen.
\newblock Splitting schemes and segregation in reaction cross-diffusion
  systems.
\newblock {\em SIAM Journal on Mathematical Analysis}, 50(5):5695--5718, 2018.

\bibitem{CintiOtto2016}
E.~Cinti and F.~Otto.
\newblock Interpolation inequalities in pattern formation.
\newblock {\em Journal of Functional Analysis}, 271(11):3348--3392, 2016.

\bibitem{Cohen_2003}
A.~Cohen, W.~Dahmen, I.~Daubechies, and R.~DeVore.
\newblock Harmonic analysis of the space {BV}.
\newblock {\em Revista Matem{\'{a}}tica Iberoamericana}, pages 235--263, 2003.

\bibitem{CraigKimYao2018}
K.~Craig, I.~Kim, and Y.~Yao.
\newblock Congested aggregation via {N}ewtonian interaction.
\newblock {\em Arch. Ration. Mech. Anal.}, 227(1):1--67, 2018.

\bibitem{DavidPerthame2021}
N.~David and B.~Perthame.
\newblock Free boundary limit of a tumor growth model with nutrient.
\newblock {\em Journal de Math{\'e}matiques Pures et Appliqu{\'e}es}, 2021.

\bibitem{DavidSchmidtchen2021}
N.~David and M.~Schmidtchen.
\newblock On the incompressible limit for a tumour growth model incorporating
  convective effects.
\newblock {\em arXiv:2103.02564}, 2021.

\bibitem{DebiecEtAl2021}
T.~D{\k e}biec, B.~Perthame, M.~Schmidtchen, and N.~Vauchelet.
\newblock Incompressible limit for a two-species model with coupling through
  {B}rinkman's law in any dimension.
\newblock {\em Journal de Math{\'e}matiques Pures et Appliqu{\'e}es},
  145:204--239, 2021.

\bibitem{DebiecSchmidtchen2020}
T.~D{\k e}biec and M.~Schmidtchen.
\newblock Incompressible limit for a two-species tumour model with coupling
  through {B}rinkman's law in one dimension.
\newblock {\em Acta Applicandae Mathematicae}, 169(1):593--611, 2020.

\bibitem{DegondHua2013}
P.~Degond and J.~Hua.
\newblock Self-organized hydrodynamics with congestion and path formation in
  crowds.
\newblock {\em Journal of Computational Physics}, 237:299--319, 2013.

\bibitem{GilQuiros2001}
O.~Gil and F.~Quir{\'o}s.
\newblock Convergence of the porous media equation to hele-shaw.
\newblock {\em Nonlinear Analysis: Theory, Methods \& Applications},
  44(8):1111--1131, 2001.

\bibitem{GilQuiros2003}
O.~Gil and F.~Quir{\'o}s.
\newblock Boundary layer formation in the transition from the porous media
  equation to a {H}ele--{S}haw flow.
\newblock {\em Annales de l'Institut Henri Poincar{\'e} C, Analyse non
  lin{\'e}aire}, 20(1):13--36, 2003.

\bibitem{GKM2020}
N.~Guillen, I.~Kim, and A.~Mellet.
\newblock A {H}ele-{S}haw limit without monotonicity.
\newblock {\em ArXiv preprint arXiv: 2012.02365}, 2020.

\bibitem{GwiazdasPerthame}
P.~Gwiazda, B.~Perthame, and A.~{Świerczewska-Gwiazda}.
\newblock A two-species hyperbolic–parabolic model of tissue growth.
\newblock {\em Communications in Partial Differential Equations},
  44(12):1605--1618, 2019.

\bibitem{HechtVauchelet2017}
S.~Hecht and N.~Vauchelet.
\newblock Incompressible limit of a mechanical model for tissue growth with
  non-overlapping constraint.
\newblock {\em Communications in Mathematical Sciences}, 15(7):1913--1932,
  2017.

\bibitem{KPW2019}
I.~Kim, N.~Požár, and B.~Woodhouse.
\newblock Singular limit of the porous medium equation with a drift.
\newblock {\em Advances in Mathematics}, 349:682--732, 2019.

\bibitem{Kpo}
I.~Kim and N.~Po\v{z}\'ar.
\newblock Porous medium equation to {H}ele-{S}haw flow with general initial
  density.
\newblock {\em Trans. Amer. Math. Soc.}, 370(2):873--909, 2018.

\bibitem{LiuXu2021}
J.-G. Liu and X.~Xu.
\newblock Existence and incompressible limit of a tissue growth model with
  autophagy.
\newblock {\em arXiv:2102.03844v3}, 2021.

\bibitem{MelletPerthameQuiros2017}
A.~Mellet, B.~Perthame, and F.~Quir{\'o}s.
\newblock A {H}ele--{S}haw problem for tumor growth.
\newblock {\em Journal of Functional Analysis}, 273(10):3061--3093, 2017.

\bibitem{PerrinZatorska2015}
C.~Perrin and E.~Zatorska.
\newblock Free/congested two-phase model from weak solutions to
  multi-dimensional compressible navier-stokes equations.
\newblock {\em Communications in Partial Differential Equations},
  40(8):1558--1589, jun 2015.

\bibitem{PQV}
B.~Perthame, F.~Quir{\'o}s, and J.~L. V{\'a}zquez.
\newblock The {H}ele-{S}haw asymptotics for mechanical models of tumor growth.
\newblock {\em Arch. Ration. Mech. Anal.}, 212(1):93--127, 2014.

\bibitem{PV2015}
B.~Perthame and N.~Vauchelet.
\newblock Incompressible limit of a mechanical model of tumour growth with
  viscosity.
\newblock {\em Philos. Trans. Roy. Soc. A}, 373(2050):20140283, 16, 2015.

\bibitem{SantambrogioOTAM}
F.~Santambrogio.
\newblock {\em Optimal transport for applied mathematicians}.
\newblock Progress in Nonlinear Differential Equations and Their Applications.
  Birkh\"{a}user, NY, 2015.

\bibitem{TVCVDP}
M.~Tang, N.~Vauchelet, I.~Cheddadi, I.~Vignon-Clementel, D.~Drasdo, and
  B.~Perthame.
\newblock Composite waves for a cell population system modeling tumor growth
  and invasion.
\newblock {\em Chin. Ann. Math. Ser. B}, 34(2):295--318, 2013.

\end{thebibliography}

\end{document}